\documentclass{amsart}
\usepackage{amscd,enumerate}
\usepackage{amssymb,amsmath,amsthm}
\usepackage{latexsym,amsfonts,amscd}
\usepackage{dsfont} 

\usepackage[utf8]{inputenc}
\usepackage[OT4]{fontenc}

\newcommand*{\R}{\ensuremath{\mathbb{R}}}
\newcommand*{\Q}{\ensuremath{\mathbb Q}}
\newcommand*{\N}{\ensuremath{\mathbb{N}}}
\newcommand*{\Z}{\ensuremath{\mathbb{Z}}}
\newcommand*{\borel}{\ensuremath{\mathfrak{B}}}

\newtheorem{theorem}{Theorem}
\newtheorem{definition}[theorem]{Definition}
\newtheorem{corollary}[theorem]{Corollary}
\newtheorem{lemma}[theorem]{Lemma}
\newtheorem{proposition}[theorem]{Proposition}
\newtheorem{problem}[theorem]{Problem}

\begin{document}
\title{Notes on scale-invariance and base-invariance for Benford's Law}
\author{Michał Ryszard Wójcik}
\address{University of Wrocław, Institute of Geography and Regional Development}
\email{michal.ryszard.wojcik@gmail.com}

\begin{abstract}
It is known that if $X$ is uniformly distributed modulo 1
and $Y$ is an arbitrary random variable independent of $X$
then $Y+X$ is also uniformly distributed modulo 1.
We prove a converse for any continuous random variable $Y$
(or a reasonable approximation to a continuous random variable)
so that if $X$ and $Y+X$ are equally distributed modulo 1
and $Y$ is independent of $X$ then
$X$ is uniformly distributed modulo 1
(or approximates the uniform distribution equally reasonably).
This translates into a characterization of Benford's law
through a generalization of scale-invariance:
from multiplication by a constant to multiplication
by an independent random variable.

We also show a base-invariance characterization:
if a positive continuous random variable has the same
significand distribution for two bases then it is Benford
for both bases. The set of bases for which a random variable
is Benford is characterized through characteristic functions.
\end{abstract}
\maketitle

\section{Introduction}
Before the early 1970s, handheld electronic calculators were not yet in widespread use
and scientists routinely used in their calculations books with tables containing the decimal logarithms of
numbers between 1 and 10 spaced evenly with small increments like 0.01 or 0.001.
For example, the first page would be filled with numbers 1.01, 1.02, 1.03, \ldots, 1.99
in the left column and their decimal logarithms in the right column,
while the second page with 2.00, 2.01, 2.02, 2.03, \ldots, 2.99, and so on
till the ninth page with 9.00, 9.01, 9.02, 9.03, \ldots, 9.99,
with the decimal logarithms in the right columns increasing from 0 to 1 throughout the nine pages.

Back in 1881, astronomer and mathematician Simon Newcomb published a two-page note \cite{newcomb}
which started with the observation that the earlier pages in books with logarithmic tables
were more worn by use than the later pages,
giving evidence to the empirical fact that numbers occurring in the scientific analysis of nature
(which he calls {\em natural numbers})
are more likely to begin with lower digits than higher digits,
disregarding the initial zeros for numbers less than 1.
He also presented a very short heuristic argument to the effect that
for {\em natural numbers}
(that is those occurring in the scientific analysis of nature)
the fractional part of the decimal logarithms is distributed evenly in the interval between 0 and 1,
from which he derived the frequencies with which each digit 1, 2, \ldots, 9 appears
as the first or most significant in {\em natural numbers} --- or nature's numbers
not to be confused with the positive integers.
According to his heuristic rule, nature's numbers begin with the digit 1 about 30\% of the time
while the first digit is 9 only about 5\% of the time.

This subject was taken up over fifty years later, in 1937, by
electrical engineer and physicist Frank Benford who
published a long (22 pages) article in a philosophical journal \cite{benford}, where he analyzed
the frequencies of first digits for real-world data (about 20,000 numbers) which he personally collected
from diverse fields including three types of geographical data (rivers, areas, populations).
He found in each type of data a reasonable approximation to
the logarithmic distribution that was heuristically derived by Simon Newcomb
and seriously attempted to provide an explanation for what appeared to be a widely applicable statistical law of nature,
to be later known as Benford's Law.

These two early publications by Simon Newcomb and Frank Benford are a must-read for anyone interested in the subject.
In 1976, after many papers on this subject were written, 
Ralph A.~Raimi published an excellent survey \cite{raimi}
of the philosophical and mathematical efforts
to derive Benford's Law from basic principles.
In 2011, there appeared a survey of Benford's Law literature \cite{basictheory}
written by Arno Berger and Theodore P. Hill, who both have made important contributions to this field.
They also maintain a website www.benfordonline.net,
keeping track of important literature on Benford's Law.

To this day, there is no general explanation why Benford's Law should be satisfied in so many diverse
types of statistical data and probably this would have to be solved on a case by case basis
with a limited number of different types of explanations to cover all the cases.

Intriguingly, an analysis of the digit frequencies in financial data can be used to detect fraud,
as was discovered by Mark J.~Nigrini --- 
a pioneer in applying Benford's Law to auditing and forensic accounting.
See his recent book {\em Benford's Law: Applications for Forensic Accounting, Auditing, and Fraud Detection}
from 2012, where he shows the widespread applicability of Benford's Law and its practical uses
to detect fraud, errors, and other anomalies.

In the 21st century, there appeared a number of articles from broadly understood Earth and environmental sciences
(geosciences or geophysical sciences) seeking application's of Benford's Law, \cite{geo1}, \cite{geo2}, \cite{geo3},
\cite{geo4}, \cite{geo5}, \cite{geo6}.
In \cite{geo4}, geophysical observables like length of time between geomagnetic reversals, depths of earthquakes,
seismic wavespeeds and others were found to reasonably conform to the first-digit law,
and a case of earthquake detection by first-digit analysis alone was reported.

The applicability of Benford's Law to natural sciences has been a kind of mystery from the start,
with two kinds of insights being presented as its justification. If there is any law
governing the distribution of first digits in nature's numbers then it should not be sensitive to
the choice of units, so that a large table of numbers given in miles after being recalculated to kilometers
should exhibit the same digit frequencies --- this is called scale-invariance.
Alternatively, such a law should not be dependent on the particular choice of the number 10
for the basis of our numerical system, which gives rise to the notion of base-invariance.
Both notions have been used to characterize Benford's law
as the only significand distribution which satisfies any of them.

This article is strictly mathematical, written with the sole purpose of reviewing the basic mathematics
behind the notions of scale-invariance and base-invariance.
Relevant known results are discussed: sometimes simplified and sometimes strengthened
to widen their applicability and to achieve greater elegance and clarity.

Scale-invariance is highly extended so that multiplication by
a constant is replaced with multiplication by any independent
continuous random variable. Scale-invariance being dependent
on the choice of base, dependence on the base is studied in detail
with the final conclusion that the perceived base-invariance
of Benford's Law as a stistical law of nature is due to
the approximating nature of limiting processes 
rather than exact because in the most natural cases Benford's law depends
on the choice of base and is not fully base-invariant.

Although the very first article by Simon Newcomb was written with full awareness of the arbitrary role
of the number 10 and although his heuristic argument explains the distribution of the logarithms
with no special role being played by digits, it has become usual for later authors to focus
on the first (or first two) decimal digits of numbers rather than on the distribution of their significand.
After all, grouping data into nine slots according to the first digit is just one arbitrary way
of making a histogram for the distribution of the significand.
Considering first two digits is just making a finer histogram.
What should really be studied is the distribution of the significand for many different bases,
not necessary being integers. These and other terminological matters are covered in the following preliminary section.

\subsection*{Notational conventions}
Throughout the article, $\N$ is the set of positive integers without zero and $\Z$ is the set of all integers,
while $\lambda$ is the Lebesgue measure.
For a random variable $X$, its characteristic function
is denoted $\varphi_X$.

\section{Terminological background for significand analysis}
Let us collect the basic facts and notational conventions
needed for an analysis of the distribution of the significand
(usually called the mantissa) of a random variable.

Let us fix a base $b\in(1,\infty)$.
The significand $S_b(x)$ of a positive number $x>0$  to base $b$ is
the unique element of the sequence
$$x,xb,xb^{-1},xb^2,xb^{-2},\ldots$$
which belongs to $[1,b)$. Explicitly,
\begin{equation}\label{significand}S_b(x)=xb^{-\lfloor\log_bx\rfloor},\end{equation}
where $$\lfloor t\rfloor=\max\{n\in\Z\colon n\le t\}$$
so that for each $t\in\R$ we have $$\lfloor t\rfloor\in\Z
\text{ and }\lfloor t\rfloor\le t<\lfloor t\rfloor+1.$$
Indeed,
$$0\le\log_bx-\lfloor\log_bx\rfloor<1\iff
b^0\le b^{\log_bx-\lfloor\log_bx\rfloor}<b^1\iff
1\le S_b(x)<b.$$
The significand is usually called the mantissa
in the literature on Benford's Law.
See the note after Definition 2.3 in \cite{basictheory}
for a discussion of the reasons why a change of terminology
is forthcoming.

For psychological reasons people often restrict their analysis
to base 10 and when they do consider other bases
they usually think of other integers. However,
any number $b\in(1,\infty)$ is good to be considered
as the base in significand analysis
See \cite{baseb} for an analysis of conformance
to Benford's Law in dependence on the base
or \cite{scaleinvariant} for a description of the set of all bases
in which a given random variable satisfies Benford's law.
See also \cite{schatte1981}.

A positive random variable $X\colon\Omega\to(0,\infty)$
is Benford (base $b$) (cf. \cite[Def.~3.4]{basictheory})
if \begin{equation}\label{d1}
P(S_b(X)\le t)=\log_bt\text{ \ \ \ for all }t\in[1,b)\end{equation}
or equivalently \begin{equation}\label{d2}
P(S_b(X)\le b^t)=t\text{ \ \ \ for all }t\in[0,1)\end{equation}
or equivalently \begin{equation}\label{d3}
\log_bS_b(X)\sim U[0,1)\end{equation}
or equivalently without the significand notation
\begin{equation}\label{d4}\log_bX-\lfloor\log_bX\rfloor\sim U[0,1)\end{equation}
or equivalently using the modulo notation: $x\bmod{1}:=x-\lfloor x\rfloor$
\begin{equation}\label{d5}(\log_bX)\bmod{1}\sim U[0,1).\end{equation}
The last line justifies why the study of the Benford property
can be conveniently carried out in the context of
random variables whose distribution is concentrated on $[0,1)$
with arithmetic operations considered modulo 1,
which is adopted throughout this article.

It is safer to talk about the Benford property of a random variable $X$
rather than talk in terms of $X$ satisfying Benford's Law
or the distribution of $X$ being Benford because
the Benford property determines the distribution of
the random variable $S_b(X)$ and not $X$ itself.
For example, if $U\sim U[c,d)$ with $d-c\in\N$
then each random variable of the form $b^U$
has the Benford property base $b$
but they all have different distributions.
Note that the density $f$ of such a $b^U$ is given by
$$f(x)=\frac{1}{(d-c)\ln b}\mathds{1}_{[b^c,b^d)}(x)
\frac{1}{x}.$$

In the English mathematical language, the words {\em law}
and {\em distribution} are used interchangeably,
which may be misleading because the word {\em law}
in natural sciences refers to so called laws of nature.
It should be clearly differentiated whether we are talking about
the observation (which may be referred to as Benford's Law)
that in many statistical datasets the empirical
distribution of the significand is close to having
the Benford property on the one hand,
and the mathematical fact that a particular random variable
has the Benford property on the other.
Therefore, I propose to write {\em Benford's Law}
with a capital L to refer to the observation concerning
the statistics collected by humans (a kind of law of nature)
and to write {\em Benford's law} with a lower-case l
to refer to this kind of distribution as a mathematical object.

\subsection*{A note on first digits}
For $b>1$, the interval $[1,b)$ can be divided into
a number of equal parts
as a basis for a histogram for a collection of numbers
lying between 1 and $b$, which is the case when we consider
the significand $S_b(X)$ for any random variable $X$
or a statistical collection of numbers.

If $b$ is an integer, we may have $b-1$ equal parts
$[1,2), [2,3), \ldots, [b-1,b)$.
With this kind of histogram we investigate the frequencies
of first digits in the dataset as we study the distribution
of the significand for the integer base $b$.
If Benford's law is satisfied for base $b\in(1,\infty)$, then
$$P(s\le S_b(X)<t)=\log_bt-\log_bs=\log_b\frac{t}{s}
\text{ \ \ for }1\le s<t\le b.$$
If $b=n\in\N$,
the frequency for a first digit $d\in\{1,2,\ldots,n-1\}$
is then given by $$P(d\le S_n(X)<d+1)=\log_n\frac{d+1}{d}
=\log_n\Big(1+\frac{1}{d}\Big),$$
which is the most familiar formulation of Benford's Law,
also known as the first-digit law.

It should be kept in mind that the choice of base is arbitrary
and there is no mathematical reason to consider only integers.
Any choice for a histogram can be made,
not necessarily based on first digits, or first two digits.

\section{Overview of the results}

\subsection*{The scale-invariance section}
Recall the definition of the significand to base $b$,
$S_b(x)=xb^{-\lfloor\log_bx\rfloor}$, (\ref{significand}).
For a positive real random variable $X\colon\Omega\to(0,\infty)$
and a fixed base $b\in(1,\infty)$ we study the distribution
of the significand $S_b(X)$.
Whether a random variable can have the Benford property
for two or more distinct bases at the same time
is discussed in the secion on base-invariance.

We will write $X$ is $b$-Benford as a shorthand for $\log_b(S_b(X))\sim U[0,1)$,
which is the Benford property for base $b$, cf.~(\ref{d1})-(\ref{d5}).
The following points summarize the key facts concerning scale-invariance
and Benford's law:
\begin{equation}\label{f1}
X\text{ is }b\text{-Benford }\implies aX\text{ is }b\text{-Benford for all }a>0
\end{equation}\begin{equation}\label{f2}
S_b(X)\sim S_b(aX)\text{ for some }a\text{ with }
\log_ba\not\in\Q\implies X\text{ is }b\text{-Benford}
\end{equation}\begin{equation}\label{f25}
S_b(X)\sim S_b(aX)\text{ for infinitely many constants }a
\implies X\text{ is }b\text{-Benford}
\end{equation}\begin{equation}\label{f3}
X\text{ is }b\text{-Benford}\implies YX\text{ is }b\text{-Benford}
\text{ for any independent }Y>0
\end{equation}\begin{equation}\label{f4}
S_b(X)\sim S_b(YX)\text{, }Y\text{ is continuous, independent of }X
\implies X\text{ is }b\text{-Benford}.
\end{equation}

Finally, the main contribution of this section, Theorem \ref{final},
is that if $Y$ is a reasonable approximation to a continuous variable then
\begin{equation}\label{f5}
S_b(X)\sim S_b(YX)\text{ and }Y\text{ has }N\text{ atoms and is independent of }X
\end{equation}
implies that $X$ is reasonably close to being $b$-Benford:
\begin{equation}\label{f6}
\sup_{t\in[1,b)}\big|P(S_b(X)\le t)-\log_bt\big|\le\frac{1}{N}.
\end{equation}

References for these points are provided:
\begin{enumerate}[ ]
\item (\ref{f1}) \cite[Th.~3]{sarkar}, Theorem \ref{filozof};
\item (\ref{f2}) \cite[Th~3.8]{baseinvariance},
\cite[Th.~4.13(ii)]{basictheory}, Theorem \ref{scalesummary};
\item (\ref{f25}) \cite{pinkham}, \cite{scaleinvariant},
Theorem \ref{gdtygfeexed};
\item (\ref{f3}) \cite[Satz~4.6]{verteilung}, \cite[Th.~4.13(i)]{basictheory},
\cite[p.~883]{appliedfourier}, \cite[IV]{hamming},
\cite[p.~64 (8.7)]{feller}, Theorem \ref{filozof};
\item (\ref{f4}) \cite[Satz 4.7]{verteilung}, a comment to \cite[Th.~2.2]{sumsmodulo}, Theorem \ref{scalesummary};
\item (\ref{f5})$\implies$(\ref{f6}) Theorem \ref{final}.
\end{enumerate}

Benford's law can be characterized as
the unique distribution of the significand that is invariant under multiplication
by an independent continuous variable.
A different characterizations of Benford's law is given in \cite{invariantsum}.
Benford's law can also be characterized by \cite[Th.~3.5]{baseinvariance}, which will be discussed in the base-invariance section, Theorem \ref{hills}.

\subsection*{The base-invariance section}
This time the focus is on the set of those bases $b>1$
for which a given positive random variable $X$ is $b$-Benford,
denoted $B_X$ and called the Benford spectrum of $X$:
\begin{equation}\label{definitionspectrum}
B_X=\{b\in(1,\infty)\colon X\text{ is }b\text{-Benford}\}.
\end{equation}
The following points summarize the key facts about
Benford spectra. Due to the identity
\begin{equation}S_b(x^a)=S_{b^{1/a}}(x)
\text{ \ \ for all }b>1,a>0\end{equation}
they are sometimes stated in two different but equivalent ways:
\begin{equation}\label{s1}B_X\text{ is bounded}\end{equation}
\begin{equation}\label{s2}b^{1/a}\in B_X\iff
b\in B_{X^a}\end{equation}
\begin{equation}\label{sq1}b\in B_X\implies\sqrt[n]{b}\in B_X
\text{ \ \ for all }n\in\N
\end{equation}\begin{equation}\label{sq2}
b\in B_X\implies{b}\in B_{X^n}
\text{ \ \ for all }n\in\N
\end{equation}\begin{equation}\label{hilll1}
P(S_b(X)=1)=0\text{ and }
S_b(X)\sim S_b(X^n)\text{ for all }n\in\N\implies b\in B_X
\end{equation}\begin{equation}\label{hilll2}
P(S_b(X)=1)=0\text{ and }
S_b(X)\sim S_{\sqrt[n]{b}}(X)\text{ for all }n\in\N\implies b\in B_X
\end{equation}\begin{equation}\label{etr64tef}
X\text{ has a density and }S_b(X)\sim S_\beta(X),1<b<\beta
\implies b,\beta\in B_X
\end{equation}\begin{equation}\label{gtftftccxgf}
X\text{ has a density }\implies
S_b(X^a)\to b^{U[0,1)}\text{ in distribution as }a\to\infty
\end{equation}\begin{equation}\label{gttfvccyduffjhg}
b\in B_X\iff\varphi_{\ln X}\Big(\frac{2\pi n}{\ln b}\Big)=0
\text{ \ \ for all }n\in\N
\end{equation}\begin{equation}\label{fde5ef5d}
(\forall b>1)(\exists X)\ B_X=(1,b]
\end{equation}\begin{equation}\label{gvcgvdjnki}
X=b^{U[c,d)}\implies B_X=
\big\{b^{\frac{d-c}{n}}\colon n\in\N\big\}
\end{equation}\begin{equation}\label{re4ere5er}
X,Y\text{ are independent}\implies B_X\cup B_Y\subset B_{XY}
\end{equation}

References for these points are provided:
\begin{enumerate}[ ]
\item (\ref{s1}) \cite{schatte1981},
\cite[p.~392]{baseb}, \cite[Lemma 1]{scaleinvariant},
\cite{lolbert},\\
Theorem \ref{boundedspectrum};
\item (\ref{sq1}), (\ref{sq2}) \cite[Lemma 1]{scaleinvariant},
Theorem \ref{hut6tfdfecbh};
\item (\ref{hilll1}), (\ref{hilll2}) \cite{baseinvariance},
\cite[Sec.~4.3]{basictheory}, Theorem \ref{hills};
\item (\ref{etr64tef}) Theorem \ref{baseinvchar};
\item (\ref{gtftftccxgf}) \cite{lolbert},
Corollary \ref{powers};
\item (\ref{gttfvccyduffjhg}) \cite[Th.~4]{scaleinvariant}, \cite{lolbert}, 
Theorem \ref{charbases};
\item (\ref{fde5ef5d}) \cite{schatte1981}, \cite[last page]{scaleinvariant},
Theorem \ref{intervalspectrum};
\item (\ref{gvcgvdjnki}) Theorem \ref{bhcbvtgtgve};
\item (\ref{re4ere5er}) \cite[last page]{scaleinvariant},
Theorem \ref{gsgsgsgsgs};
\end{enumerate}

The main contribution of this section is the base-invariance
characterization of Benford's law (\ref{etr64tef}),
which is derived from (\ref{gtftftccxgf}).
Another one is the fully general form of
(\ref{gttfvccyduffjhg}) and (\ref{re4ere5er}),
which required an additional assumption
of density with bounded variation 
in the original formulation by James V.~Whittaker
in \cite{scaleinvariant}.
The proofs were simplified and cleared of this surplus assumption.
Twenty five years after Whittaker,
Tamás Lolbert proved (\ref{gttfvccyduffjhg})
without any reference to his predecessor,
using the assumption of having a density
without the need for bounded variation.
Still, our Theorem \ref{charbases} is more general.

(\ref{hilll1}) or equivalently (\ref{hilll2})
is the famous characterization of Benford's law
due to Theodore P.~Hill, see Theorem \ref{hills}.
From (\ref{etr64tef})
with $\beta=b^n$, $n\in\N,n>1$ it can be modified
so that only one $n$ is required at the cost of
assuming additionally that $X$ has a density,
which is probably a known result in the context of
Furstenberg's Conjecture.

\section{Scale-invariance}
Scale-invariance is one of the most intriguing aspects of the Benford property. It is well-known
(e.g. \cite[Th.~3]{sarkar} from 1968) that
if $X\colon\Omega\to(0,\infty)$ is Benford base $b\in(1,\infty)$,
then for any constant $a>0$ the random variable $aX$ is also Benford base $b$.
This theorem becomes less mysterious when viewed as a consequence
of the fact that
\begin{equation}
Y\sim U[0,1)\implies
(c+Y)\bmod{1}\sim U[0,1)\text{ \ \ for any }c\in\R,
\end{equation}
which is simply the translation invariance of the Lebesque measure
on $[0,1)$ modulo 1.
In our case, $Y=\log_bX$ and $c=\log_ba$.

Now that we are focused on the Lebesgue measure on $[0,1)$
we can speculate about a converse theorem:
\begin{equation}\label{converse}
Y\sim(c+Y)\bmod{1}\text{ for some }c\in\R
\implies Y\sim U[0,1),\end{equation}
which translates into
$$S_b(aX)\sim S_b(X)\text{ for some }a\in(0,\infty)\implies
X \text{ is Benford base }b.$$
Notice that if $c=m/n$ with $m,n\in\N$
then, for example,
any distribution on $[0,1)$ having a density with period $1/n$
remains the same after translation by $c$ modulo 1.
This means that the converse does not hold for rational $c$
and thus for $\log_ba\in\Q$.
The proof of the converse \cite[Th.~4.13(ii)]{basictheory}
in the irrational case
is very short but it involves elements of Fourier analysis
so it is not trivial at all.

In my quest to understand why Benford's Law shows up in natural
science data I design mathematical models which are meant to simulate
numerical aspects of natural phenomena and run these simulations
to obtain large collections of numbers which I study for conformance
to Benford's Law. I always multiply each such collection by a constant
to see how it affects the distribution of the significand.
Not surprisingly, datasets which conform to Benford's Law
invariably preserve this property after multiplication by a constant.
However, I once thought of the crazy idea of multiplying
such a collection by a random set of numbers, so that each element
is multiplied by a different number coming from some arbitrary distribution.
One could expect that anything at all can come out of this
but to my surprise Benford's Law was always preserved.
In this way I discovered a generalization of scale-invariance for Benford's Law
for a fixed base $b\in(1,\infty)$:
\begin{equation}\label{XYX}
X\text{ is Benford and }Y\text{ is independent of }X\implies
YX\text{ is Benford}.\end{equation}
I immediately thought of a probability-theory explanation
of the empirical facts of my simulations along these lines:
if $Y$ takes only two values $a$ and $b$ then the original Benford dataset
is split into two subsets $A$ and $B$ so that
$A$ contains those numbers which are multiplied by $a$
and $B$ those which are multiplied by $b$.
Since $Y$ is independent of $X$, these two subsets remain Benford.
Each of these subsets remains Benford after multiplication
and so their mixture is still Benford.
This can be extended first to any random variable
with finitely many values
and then to any random variable
by limiting processes typical in measure theory.
Although I wrote a formal proof along these lines,
it was still a mystery how it could explain the empirical facts
of the simulations. After all, the argument works well only
when the original dataset is split into subsets which are large enough
to be reasonably thought of as Benford sets. The fact of the matter
is that very probably these subsets were singletons in the actual simulations
because I always chose random numbers from continuous distributions.
My friend who is an excellent mathematician managed to explain
the phenomenon in general philosophical terms so that
the mystery was adequately solved but we must stick
with the formalism of probability theory in this article.
Perhaps something of the spirit of the philosophical argument
is captured in the formalism of the proof of the following theorem
--- which is a restatement of (\ref{XYX}) into the perspective
of $[0,1)$ modulo 1.

\begin{theorem}\label{filozof}
Let $X\sim U[0,1)$. If $Y$ is independent of $X$
then $(Y+X)\bmod{1}\sim U[0,1)$.
\end{theorem}
\begin{proof}
This proof is nothing more than an adaptation of the standard
convolution of two distributions so we only write the essentials
without any comments.
$$\mu_{Y+X}(A)=\int_{\R^2}\mathds{1}_Ed\mu_{(X,Y)}
=\int_\R\Big(\int_\R\mathds{1}_{A-y}(x)d\mu_X(x)\Big)d\mu_Y(y)$$
$$=\int_\R\Big(\mu_X(A-y)\Big)d\mu_Y(y)
=\int_\R\Big(\mu_X(A)\Big)d\mu_Y
=\mu_X(A)\int_\R d\mu_Y=\mu_X(A)$$

$$\mu_X(A)=P(X\in A),\ \mu_Y(A)=P(Y\in A)$$
$$\mu_{Y+X}(A)=P((Y+X)\bmod{1}\in A)$$
$$\mu_{(X,Y)}(E)=P((X,Y)\in E)$$
$$E=\{(x,y)\colon(x+y)\bmod{1}\in A\}$$
$$A-y=\{x\colon(x+y)\bmod{1}\in A\}$$
\end{proof}

After digging in the literature I found this theorem
in a recent 2011 survey of Benford's Law theory
\cite[Th.~4.13(i)]{basictheory}
by Arno Berger and Theodore P. Hill,
where it is proven using the fact that Fourier coefficients
uniquely determine any distribution on $[0,1)$
and similarly in \cite[Satz 4.6]{verteilung} and \cite[page 883]{appliedfourier}.
A proof without using Fourier coefficients
for random variables having densities
is given in \cite[IV]{hamming}.
An analogous proof of Theorem \ref{filozof} is given
in Feller's textbook \cite[p.~64 (8.7)]{feller}
using the convolution of two densities,
which imposes an unnecessary requirement on $Y$
that is has a density.

Note that the converse to this theorem runs into the same
difficulties as mentioned earlier (\ref{converse}).
Let us ask for which random variables $Y$ we have
\begin{equation}\label{questio}
X\sim(Y+X)\bmod{1}\text{ and }X,Y\text{ are independent }
\implies X\sim U[0,1).\end{equation}
The answer is that it works as long as
$P(Y\in\Q)<1$ or
$Y$ modulo 1 has infinitely many rational atoms in $[0,1)$.
The excluded case is when $Y$ is concentrated
on an arithmetic sequence of rational numbers.
For bounded variables it is the same as
assuming finitely many rational values
with probability one,
that is $P(Y\in A)=1$ for some finite $A\subset\Q$.

All the ingredients of the proof can be found in the survey just-mentioned
but it is not given there explicitly --- see the details of the proofs
of the three parts of \cite[Th.~4.13]{basictheory}.
An explicit statement with a proof can be found
in a 1973 paper written in German by Peter Schatte,
\cite[Satz 4.7]{verteilung}. He also cites it again in 1983
as a comment to \cite[Th.~2.2]{sumsmodulo}.
Here it is given as Theorem \ref{scalesummary}.

Beside this compilatory effort,
my contribution is the answer to the question
what can be inferred about $X$ when
$Y$ is independent of $X$
and $Y+X$ is distributed the same as $X$ modulo 1
in the general case when $Y$ is permitted to be
concentrated on a finite set of rationals.
It turns out that $X$ can be reasonably close
to being uniformly distributed modulo 1
as long as $Y$ is a reasonable approximation
of a continuous random variable.

My key result in this section (Theorem \ref{mybest}) is that if
$$X\sim(Y+X)\bmod{1}$$
and $$Y\text{ is independent of }X$$ with
$$N=\min\Big\{n\in\N\colon P\big(nY\in\Z\big)=1\Big\}<\infty$$
then $$X\sim\Big(\frac{1}{N}+X\Big)\bmod{1},$$
and thus $$X\text{ is equally distributed on each interval }
\Big[\frac{k}{N},\frac{k+1}{N}\Big), k=0,1,\ldots,N-1.$$

I use the Fourier series technique inspired by the German article by Peter Schatte
\cite{verteilung},
which is written in the Riemann–Stieltjes integration paradigm
with cumulative distribution functions being continuous from the left.
Because of these technicalities (and the fact of being written in German
with different notational conventions) I decided to rewrite the arguments
and put them into the Lebesgue integration paradigm
with continuous from the right distribution functions,
which yielded more elegant formulations of the key formulas
involving expansion into Fourier series.

The result is a self-contained exposition of the kind of Fourier analysis
mentioned earlier plus additional techniques of studying cumulative distribution functions through their Fourier coefficients.
The only external reference is the following well-known Jordan's criterion
for the pointwise convergence of a Fourier series,
stated in the simplest possible version for our purposes.

\begin{theorem}[Jordan's Test]\label{fourierpointwise}
If $F\colon[0,1)\to[0,1]$ is monotonic then
$$\frac{F(x^+)+F(x^-)}{2}=\lim_{N\to\infty}\sum_{n=-N}^{n=N}
\hat{F}(n)e^{2\pi inx}$$
for all $x\in(0,1)$, where
$$\hat{F}(n)=\int_0^1F(t)e^{-2\pi int}dt$$
for each $n\in\Z$ and $F(x^+)$, $F(x^-)$ denote the one-sided limits of $F$ at $x$.
\end{theorem}
\begin{proof}
Apply \cite[10.1.1]{fourier} to $f\colon[0,2\pi)\to\R$ given by $f(x)=F(x/2\pi)$.
\end{proof}

\subsection*{Fourier coefficients}
If $f\colon[0,1)\to\R$ is integrable, let
\begin{equation}\label{fourdef}
\hat{f}(n)=\int_0^1f(t)e^{-2\pi int}dt\text{ \ \ for each }n\in\Z,
\end{equation}
cf.~\cite[(4)]{appliedfourier}.
The numbers $\hat{f}(n)$ are called the Fourier coefficients of
the function $f$.
Note that if $f$ is the density of a random variable
$X\colon\Omega\to[0,1)$ and 
$\varphi(t)=Ee^{itX}$ is the characteristic function of $X$,
then $\hat{f}(n)=\varphi(-2\pi n)$ for each $n\in\Z$
and $\varphi(0)=1$.
The numbers $\varphi(-2\pi n)$ are called the Fourier coefficients
of the distribution of the random variable $X$
even when it has no density.

Let $X,Y\colon\Omega\to\R$ be random variables
on a probability space $(\Omega,\mathfrak M,P)$.
Then $$\varphi_X(t)=Ee^{itX}=
\int_\Omega e^{itX(\omega)}dP(\omega)
=\int_\R e^{itx}d\mu_X(x),$$
where $\mu_X(A)=P(X\in A)$,
and analogously $\varphi_Y$ are their characteristic functions.
If $X$ and $Y$ are independent, then
--- as a consequence of Fubini's theorem,
\begin{equation}\label{phixy}
\varphi_{X+Y}(t)=Ee^{it(X+Y)}=Ee^{itX}\cdot
Ee^{itY}=\varphi_X(t)\varphi_Y(t).\end{equation}
The following three facts will be repeatedly used in
both this section and the next.
\begin{equation}\label{modulosmarts}
\varphi_X(2\pi n)=Ee^{2\pi inX}=
Ee^{2\pi in(X\bmod{1})}=\varphi_{X\bmod{1}}(2\pi n)
\text{ \ \ for all }n\in\Z.\end{equation}
As a consequence we have for independent random variables
$X,Y\colon\Omega\to\R$
\begin{equation}\label{independentsmarts}
\varphi_{(X+Y)\bmod{1}}(2\pi n)=\varphi_X(2\pi n)\varphi_Y(2\pi n)
\text{ \ \ for all }n\in\Z.\end{equation}
The uniform distribution $U[0,1)$ has very simple Fourier coefficients:
\begin{equation}\label{u01}
\varphi_{U[0,1)}(2\pi n)=\int_0^1e^{2\pi in}dx=0
\text{ \ \ for all }n\in\Z\setminus\{0\}.\end{equation}

The application of Fourier series to the study of Benford's Law
is excellently discussed in \cite{appliedfourier}
under the supposition that the random variables under consideration
have densities in $L^2$. Unfortunately, this is not enough for us
because we want to characterize Benford's Law
as the only distribution satisfying (\ref{XYX})
in the family of all distributions.
Therefore we look at the Fourier coefficients
of cumulative distribution functions rather than densities.

The following lemma will be used to express the Fourier coefficients
of a cumulative distribution function in terms of
the Fourier coefficients of the distribution itself.
Integration by parts of Stieltjes integrals (as practiced by Peter Schatte
in this context)
 is thus avoided
and substituted with the application of Fubini's Theorem
in the usual manner of measure theory based on Lebesgue integrals.

\begin{lemma}\label{nosteel}
Let $\mu$ be a probability measure on the Borel subsets of $[0,1)$.
If $h\colon[0,1]\to\R$ is continuously differentiable, then
$$\int_0^1 h'(t)\mu([0,t])dt
=h(1)-\int_{[0,1)}h(t)d\mu(t).$$
\end{lemma}
\begin{proof}
$$\int_0^1 h'(t)\mu([0,t])dt=\int_0^1\Big(\int_{[0,1)}h'(t)\mathds{1}_{[0,t]}(s)d\mu(s)\Big)dt$$
$$=\int_{[0,1)}\Big(\int_0^1h'(t)\mathds{1}_{[s,1]}(t)dt\Big)d\mu(s)
=\int_{[0,1)}\Big(\int_s^1h'(t)dt\Big)d\mu(s)$$
$$=\int_{[0,1)}\Big(h(1)-h(s)\Big)d\mu(s)
=h(1)\mu([0,1))-\int_{[0,1)}h(s)d\mu(s).$$
\end{proof}

The same reasoning can be conducted
if we write $\mu([0,t))$ instead of $\mu([0,t])$
so it can also be used with the continuous from the left
version of the cumulation distribution function $P(X<t)$.

\begin{theorem}[cf.~\cite{verteilung} (3.2-5)]
\label{dystrybuanta}
If $X\colon\Omega\to[0,1)$ is a random variable,
$F(x)=P(X\le x)$ its cumulative distribution function,
and $\varphi(t)=Ee^{itX}$ its characteristic function,
then
$$\frac{F(x)+F(x^-)}{2}=1-EX+\lim_{N\to\infty}\sum_{0<|n|\le N}
\Big(\frac{\varphi(-2\pi n)-1}{2\pi in}\Big)e^{2\pi inx}$$
for all $x\in(0,1)$, where $F(x^-)$ denotes the left-sided limit of $F$ at $x$.
\end{theorem}
\begin{proof}
First apply Theorem \ref{fourierpointwise}.
Since $F$ is continuous from the right, $F(x^+)=F(x)$.
Let $\mu(A)=P(X\in A)$ for all $A\in\borel([0,1))$.
Adopt Lemma \ref{nosteel} to the complex-valued function
$h(t)=e^{-2\pi int}=\cos(-2\pi nt)+i\sin(-2\pi nt)$
to conclude that
$$\int_0^1F(t)e^{-2\pi int}dt
=\int_0^1\Big(\frac{e^{-2\pi int}}{-2\pi in}\Big)'\mu([0,t])dt
=\frac{1}{-2\pi in}-\int_{[0,1)}\frac{e^{-2\pi int}}{-2\pi in}d\mu(t),$$
which means that $\hat{F}(n)=\frac{\varphi(-2\pi n)-1}{2\pi in}$
for each $n\in\Z\setminus\{0\}$.

Apply Lemma \ref{nosteel} to $h(t)=t$ to conclude that $\hat{F}(0)=1-EX$.
\end{proof}

\begin{theorem}
[cf. XIX.6 in \cite{feller} and Th.~2.2 in \cite{sumsmodulo}]
\label{coefficients}
Let $X,Y\colon\Omega\to[0,1)$ be random variables such that
$\varphi_X(2\pi n)=\varphi_Y(2\pi n)$ for all $n\in\N$.
Then $X\sim Y$.
\end{theorem}
\begin{proof}
Notice that for any characteristic function $\varphi(-t)=\overline{\varphi(t)}$.
This means that our assumption implies that
$\varphi_X(2\pi n)=\varphi_Y(2\pi n)$ for all $n\in\Z$.
Let $F_X$, $F_Y$ be the cumulative distribution functions of $X$ and $Y$.
Since they are nondecreasing, there is a countable set $E\subset(0,1)$
such that $F_X$ and $F_Y$ are both continuous on $(0,1)\setminus E$,
which is dense in $(0,1)$.
From Theorem \ref{dystrybuanta} it follows that
$F_X(t)+EX=F_Y(t)+EY$ for all $t\in(0,1)\setminus E$
because they are expressed in terms of
$\varphi_X(2\pi n)=\varphi_Y(2\pi n)$.
Since these functions are continuous from the right,
$F_X(t)+EX=F_Y(t)+EY$ for all $t\in[0,1)$.
Since $P(X=1)=P(Y=1)=0$,
both $F_X$ and $F_Y$ are continuous from the left at $1$,
so $$1+EX=F_X(1)+EX=F_Y(1)+EY=1+EY.$$
It follows that $F_X=F_Y$ and thus $X\sim Y$.
\end{proof}

\begin{theorem}[Peter Schatte, Satz 4.7 in \cite{verteilung}]
\label{PeterSchatte}
Let $X,Y\colon\Omega\to[0,1)$ be independent random variables.
Suppose that $X\sim (X+Y)\bmod{1}$ and $\varphi_Y(2\pi n)\not=1$
for all $n\in\N\setminus\{0\}$.
Then $X\sim U[0,1)$.
\end{theorem}
\begin{proof}
By assumptions, see (\ref{independentsmarts}),
$$\varphi_X(2\pi n)=
\varphi_{(X+Y)\bmod{1}}(2\pi n)=
\varphi_X(2\pi n)\varphi_Y(2\pi n)\text{ \ \ \ for all }n\in\N.$$
It follows that
$$\varphi_X(2\pi n)=0\text{ \ \ \ for all }n\in\N\setminus\{0\}.$$
By Theorem \ref{coefficients}, $X\sim U[0,1)$.
\end{proof}

Next we explain the condition
$\varphi(2\pi n)\not=1$ used in Theorem \ref{PeterSchatte}.
The following Proposition \ref{rachuneczek} deals with
the more general case $|\varphi(2\pi n)|=1$
and its Corollary \ref{char1} deals with $\varphi(2\pi n)=1$.
These well-known folklore results
(e.g.~\cite[p.~36]{basictheory}, \cite[p.~500]{feller})
are included for the sake of
completeness because they reveal the underlying technical reasons
why a refinement is needed in order to study the excluded case
discussed in (\ref{questio}).

\begin{proposition}\label{rachuneczek}
Let $X\colon\Omega\to\R$ be a random variable
and let $\mu(A)=P(X\in A)$ be its distribution. Let $t\in(0,\infty)$ and $\theta\in[0,1)$. Then
\begin{equation}\label{teta}\int_\R e^{2\pi itx}d\mu(x)=e^{2\pi i\theta}\end{equation}
if and only if
\begin{equation}\label{Nuniform}
P\Big(X\in\big\{\frac{n+\theta}{t}\colon n\in\Z\big\}\Big)=1.\end{equation}
\end{proposition}
\begin{proof}
Suppose (\ref{teta}). Then
$$1=\int_\R e^{2\pi itx}e^{-2\pi i\theta}d\mu(x)=
\int_\R e^{2\pi i(tx-\theta)}d\mu(x).$$
Thus $$\int_\R\cos\big(2\pi(tx-\theta)\big)d\mu(x)=1$$
and $$\int_\R1-\cos\big(2\pi(tx-\theta)\big)d\mu(x)=0.$$
Since the integrated function is nonnegative,
$$1=\mu\big\{x\in\R\colon\cos\big(2\pi(tx-\theta)\big)=1\big\}$$
and consequently
$$P(tX-\theta\in\Z)=1,$$
which is equivalent to (\ref{Nuniform}).

Now, suppose (\ref{Nuniform}). Let us write $p_n=P\big(X=\frac{n+\theta}{t}\big)$.
Then $$\int_\R e^{2\pi itx}d\mu(x)=
\sum_{n\in\Z}p_ne^{2\pi it\frac{n+\theta}{t}}
=\sum_{n\in\Z}p_ne^{2\pi i(n+\theta)}
=e^{2\pi i\theta}\sum_{n\in\Z}p_n=e^{2\pi i\theta}.$$
\end{proof}

\begin{corollary}\label{char1}
Let $Y\colon\Omega\to[0,1)$ be a random variable and $\varphi$ its characteristic function. Let $N\in\N$.
Then $\varphi(2\pi N)=1$ if and only if \begin{equation}
P\Big(Y\in\big\{\frac{k}{N}\colon k\in\{0,1,\ldots,N-1\}\big\}\Big)=1.\end{equation}
\end{corollary}
\begin{proof}
Apply Proposition \ref{rachuneczek} to $t=N$ and $\theta=0$.
\end{proof}

The answer to (\ref{questio}) is summarized below.

\begin{theorem}\label{scalesummary}
Let $X,Y\colon\Omega\to\R$ be independent random variables.
Suppose that $P(Y\in\Q)<1$ or $Y\bmod{1}$ has infinitely
many atoms. Then $$YX\bmod{1}\sim X\bmod{1}\implies
X\bmod{1}\sim U[0,1).$$
\end{theorem}
\begin{proof}
By Corollary \ref{char1},
$\varphi_Y(2\pi n)\not=1\text{ for all }n\in\N\setminus\{0\}$.
Apply Theorem \ref{PeterSchatte}.
\end{proof}

We already suspect that the condition $X\sim Y+X \mod 1$
has something to do with the periodicty of the distributions.
We cannot talk about the density being periodic because
we do not want to restrict our analysis to random variables
having densities. Instead, we will think in terms of
$F(t)-t$ being periodic on $[0,1)$, where $F$ is the cumulative
distribution function.
The next three technical results prepare the stage along these lines
for the culmination in Theorem \ref{mybest}.

\begin{theorem}[cf.~\cite{verteilung} Satz 3.1]\label{minusx}
If $X\colon\Omega\to[0,1)$ is a random variable,
$F(x)=P(X\le x)$ its cumulative distribution function,
and $\varphi(t)=Ee^{itX}$ its characteristic function,
then
$$\frac{F(x)+F(x^-)}{2}-x=\frac{1}{2}-EX+\lim_{N\to\infty}\sum_{0<|n|\le N}
\Big(\frac{\varphi(-2\pi n)}{2\pi in}\Big)e^{2\pi inx}$$
for all $x\in(0,1)$, where $F(x^-)$ denotes the left-sided limit of $F$ at $x$.
\end{theorem}
\begin{proof}
Apply Theorem \ref{dystrybuanta} to conclude that
$$x=\frac{1}{2}+\lim_{N\to\infty}\sum_{0<|n|\le N}
\Big(\frac{-1}{2\pi in}\Big)e^{2\pi inx}\text{ \ \ \ for all  }x\in(0,1).$$
It is important to keep track of the fact that
$$\hat{h}(n)=\frac{\varphi(-2\pi n)}{2\pi in}
\text{ \ \ for all }n\in\Z\setminus\{0\},$$
where $h(x)=\frac{F(x)+F(x^-)}{2}-x$.
\end{proof}

Note that we are making minimal use of our only external reference
Theorem \ref{fourierpointwise}. In particular, we can do without
the fact that the coefficients of a trigonometric series
are uniquely determined.

\begin{lemma}\label{zerocoefficients}
Let $f\colon[0,1)\to[0,1]$ and $N\in\N$.
Suppose that
$$f(x)=\lim_{k\to\infty}\sum_{|n|\le k}\hat{f}(n)e^{2\pi inx}
\textnormal{ for all } x\in(0,1).$$
Then the following conditions are equivalent:
\begin{equation}f\Big(x+\frac{1}{N}\Big)=f(x)
\textnormal{ for all } x\in\Big[0,1-\frac{1}{N}\Big)\end{equation}
\begin{equation}\hat{f}(n)=0\textnormal{ whenever }n
\textnormal{ is not a multiple of }N.\end{equation}
\end{lemma}
\begin{proof}
A straightforward calculation shows that (2) $\implies$ (1).
Let $f_N\colon[0,1)\to[0,1]$ be defined by $f_N(x)=f(x+\frac{1}{N}\mod 1)$.
Calculate that $$\hat{f_N}(n)=e^{2\pi in/N}\hat{f}(n)\text{ for each }n\in\Z.$$
Let $g=f-f_N$. Then $0=\hat{g}(n)=\hat{f}(n)-\hat{f_N}(n)
=\big(1-e^{2\pi in/N}\big)\hat{f}(n)$ for all $n\in\Z$,
which shows that (1) $\implies$ (2).
\end{proof}

\begin{lemma}\label{hperiodic}
Let $X\colon\Omega\to[0,1)$ be a random variable,
$F(x)=P(X\le x)$ its cumulative distribution function,
and $\varphi(t)=Ee^{itX}$ its characteristic function.
Then the following conditions are equivalent:
\begin{equation}
F\Big(x+\frac{k}{N}\Big)=F(x)+\frac{k}{N}
\textnormal{ for all }x\in\Big[0,\frac{1}{N}\Big)
\textnormal{ and }k=1,\ldots,N-1,
\end{equation}\begin{equation}
\varphi(2\pi n)=0\textnormal{ whenever }n
\textnormal{ is not a multiple of }N.\end{equation}
\end{lemma}
\begin{proof}
Let $h(x)=(F(x)+F(x^-))/2-x$.
By Theorem \ref{minusx} and Lemma \ref{zerocoefficients},
$h$ has period $1/N$ $\iff\varphi(2\pi n)=0$ whenever $n$ is not a multiple of $N$.
Next we show that $h$ has peiod $1/N\iff$ $F(x)-x$ has period $1/N$.
Indeed, notice that all three functions $F(x),F(x^-),h(x)$
have the same points of continuity.
Therefore if $h$ has period $1/N$, then for any point
of continuity $x$, $$F(x)-x=h(x)=h(x+1/N)=F(x+1/N)-(x+1/N).$$
Since points of continuity are dense and $F$ is continuous from
the right everywhere, this extends to all of $[0,1)$.
The final trick is
$$F\Big(x+\frac{k}{N}\Big)-\Big(x+\frac{k}{N}\Big)=F(x)-x
\iff F\Big(x+\frac{k}{N}\Big)=F(x)+\frac{k}{N}.$$
\end{proof}

If $X$ is a random variable, let $\varphi_X(t)=Ee^{itX}$
be its characteristic function.

\begin{theorem}\label{mybest}
Let $X,Y\colon\Omega\to[0,1)$ be independent random variables.
Let $N$ be the smallest positive integer such that
\begin{equation}\label{oY}
P\Big(Y\in\big\{\frac{k}{N}\colon k=0,1,\dots,N-1\big\}\Big)=1.\end{equation}
Then
\begin{equation}\label{xsxy}X\sim(Y+X)\bmod{1}\end{equation}
if and only if
\begin{equation}\label{periodic}
P\Big(\frac{k}{N}<X\le\frac{k+t}{N}\Big)=P\Big(\frac{l}{N}<X\le\frac{l+t}{N}\Big)
\end{equation}
for all $t\in(0,1)$ and $k,l\in\{0,1,\ldots,N-1\}$.
\end{theorem}
\begin{proof}
Notice that since $X$ and $Y$ are independent,
see (\ref{independentsmarts}),
\begin{equation}\label{fouriercoeff}
\varphi_{(Y+X)\bmod{1}}(2\pi n)=\varphi_X(2\pi n)\varphi_Y(2\pi n)\text{ for all }n\in\N.
\end{equation}
Notice that from Corollary \ref{char1} we can conclude that
(\ref{oY}) is equivalent to
\begin{equation}\label{multiples}\varphi_Y(2\pi n)=1\iff n\text{ is a multiple of } N.\end{equation}
In view of (\ref{fouriercoeff}), $X\sim(Y+X)\bmod{1}$ is equivalent to
$$\varphi_X(2\pi n)\varphi_Y(2\pi n)=\varphi_X(2\pi n)\text{ for all }n\in\N,$$
which --- in view of (\ref{multiples}) --- is equivalent to \begin{equation}\label{zeros}
\varphi_X(2\pi n)=0\text{ whenever }n\text{ is not a multiple of }N\in\N.\end{equation}
By Lemma \ref{hperiodic}, this is equivalent to (\ref{periodic}).
\end{proof}

Finally, Peter Schatte's Theorem \ref{PeterSchatte}
and my Theorem \ref{mybest} can be unified
in the following way.

\begin{theorem}\label{final}
Let $X,Y\colon\Omega\to[0,1)$ be independent random variables.
Suppose that $X\sim(Y+X)\bmod{1}$.
Then\begin{equation}\label{howeversmall}
\sup_{t\in[0,1)}\big|P\big(X\le t\big)-t\big|\le
\frac{1}{N},\end{equation}
where $$N=\min\Big\{n\in\N\colon P\big(nY\in\Z\big)=1\Big\}.$$
In particular, $X\sim U[0,1)$ if $N=\min\emptyset=\infty$.
\end{theorem}
\begin{proof}
Suppose that $N<\infty$.
Let $F(t)=P(X\le t)$. 
From Theorem \ref{mybest} it follows that
$F(1/N)=1/N$ and that $F(t)-t$ has period $1/N$.
Hence it is enough to consider $|F(t)-t|$ for $t\in[0,1/N)$,
which is bounded above by $1/N$ since $F$ is nondecreasing
and thus between $0$ and $1/N$.
\end{proof}

Notice that however small the value in (\ref{howeversmall})
may be, the alternative measure of the distance
between the distribution of $X$, $\mu_X$, 
and the Lebesgue measure $\lambda$
$$\sup_{A\in\borel([0,1))}\big|\mu_X(A)-\lambda(A)\big|$$
is equal to 1 if $X$ is concentrated on a set of Lebesgue-measure zero
and is equal to \begin{equation}\label{diffe}
\int_{f<1}(1-f(x))dx\end{equation}
if $X$ has density $f$.
For example, the density
$$f(x)=2\sum_{k=0}^{N-1}
\mathds{1}_{\big[\frac{k}{N},\frac{k+1/2}{N}\big)}(x)$$
has period $1/N$ so that (\ref{howeversmall}) is satisfied
but the value of (\ref{diffe}) is $1/4$ for each $N\in\N$.

The number $N$ in Theorem \ref{final} can be considered
as the rank of a discrete rational-valued random variable.
In computers any continuous random variable
is simulated by such variables.
If an algorithm for generating a simulated continuous random variable
is capable of producing $n$ distinct values, then the rank
of such a discrete variable is at least $n$ but it may be much greater
if the values are not distributed evenly.
In fact, the rank is the least common multiple of the denominators
of those $n$ distinct values when they are expressed as
irreducible fractions.

If we try to apply Theorem \ref{mybest} in practice
to conclude that a given dataset conforms to Benford's law
we run into a fascinating paradox.
On the one hand,
if $X\sim(1/N+X)\bmod{1}$ for a large $N$ then Theorem \ref{mybest}
forces $X$ to be very close to $U[0,1)$.
On the other hand, for any distribution on $[0,1)$
a translation modulo 1 by a small number $1/N$
does not affect the distribution very much,
which shows that nothing at all can be concluded
from the fact that $X$ and $(1/N+X)$ are similarly distributed.

\begin{definition}\label{scaleinvariantdef}
A positive random variable $X\colon\Omega\to(0,\infty)$
is scale-invariant for a given base $b>1$if\begin{equation}
\label{def14a}S_b(X)\sim S_b(cX)\textnormal{ for all }c>0
\end{equation}or equivalently\begin{equation}\label{def14b}
\log_bX\bmod{1}\sim(a+\log_bX)\bmod{1}
\textnormal{ for all }a\in\R.\end{equation}\end{definition}

Let us ask how the set of constants in this definition
can be reduced while preserving equivalence.
Naturally, (\ref{def14a}) can do with $c\in(1,b)$
and (\ref{def14b}) can do with $a\in(0,1)$
and we already know that one irrational value
can do in (\ref{def14b}) and, by the same token,
one value for $c$ in (\ref{def14a})
provided that $\log_bc$ is irrational, Theorem \ref{scalesummary}.

Notice that each irreducible fraction $a=k/N$ in (\ref{def14b})
implies that $\log_bX\bmod{1}$ has period $1/N$ on $[0,1)$,
Theorem \ref{mybest}. So if we have a finite set $A$ of rational numbers
and assume (\ref{def14b}) with $a\in A$ then we can conclude
that $\log_bX\bmod{1}$ is distributed with period $1/N$ on $[0,1)$,
where $N$ is the least common multiple of all the denominators
of the irreducible fractions from $A$.
Consequently,  Definition \ref{scaleinvariantdef}
can do with any infinite subset of constants
because it is enough to conclude that $X$ is $b$-Benford,
which in turn implies scale-invariance for all remaining constants.

\begin{theorem}\label{gdtygfeexed}
Let $X\colon\Omega\to[0,1)$.
Let $A\subset\Q\cap(0,1)$ be infinite.
Suppose that $X\sim(a+X)\bmod{1}$ for each $a\in A$.
Then $X\sim U[0,1)$.
\end{theorem}
\begin{proof}
For any $n\in\N$ there is an $a=\frac{k}{N}\in A$ with $N\ge n$.
Let $Y=a$. By Theorem \ref{final},
$$\sup_{t\in[0,1)}\big|P\big(X\le t\big)-t\big|\le
\frac{1}{N}.$$
So the supremum is actually zero.
\end{proof}

Let us compare Theorem \ref{gdtygfeexed}
with Whittaker's scale-invariance theorem
\cite[Corollary to Theorem 1]{scaleinvariant}.
He also assumes scale-invariance for an infinitely countable set
but it has to have a certain specified form based on an irrational number.
This is clearly a different approach.

\section{Base-invariance}

In mathematical folklore, Benford's Law is the observation
that in various large collections of data,
whether from natural sciences or from socio-economic statistics
or even collected randomly from various mixed sources,
the logarithms base 10 of the recorded numbers
seem to be uniformly distributed modulo 1.
It is natural to expect that this should not be the unique property of the number 10
but rather should be similarly well satisfied for base, say, 8.
This is the origin of the notion of base-invariance for Benford's Law.

It is an entirely different matter
whether a given random variable (which is a specific mathematical object)
satisfies Benford's law for base 10 and base 8 simultaneously.
For example, if $U$ is uniformly distributed on $[0,1)$
then the random variable $10^U$ satisfies Benford's law for base 10
but not for base 8.

In this section we study the set of all those bases for which
a given random variable satisfies Benford's law
and try to establish base-invariance characterizations for Benford's law
where it is assumed that a given variable
has the same significand distribution for two or more distinct bases
(with additional assumptions) and
concluded that it must be Benford for those bases.
An example, due to James V.~Whittaker,
is given of a continuous random variable that satisfies Benford's law
for all bases in the interval $(1,10]$.

We will write that a strictly positive random variable $X$ is $b$-Benford iff
\begin{equation}\label{bbenford}
\log_bX\bmod{1}\sim U[0,1),\end{equation}
where $b\in(1,\infty)$ is any base, not necessarily an integer.
Equivalently, $X$ is $b$-Benford if and only if
\begin{equation}\label{bZ}
X=b^Z\text{ where }Z\bmod{1}\sim U[0,1).\end{equation}
Thus examples of random variables with the Benford property are provided
by constructing appropriate variables with uniform distribution modulo 1.
For example:
\begin{equation}\text{the uniform distribution }
U[c,d)\text{ with }d-c\in\N\end{equation}
\begin{equation}\text{the density }
\sum_{n\in\Z}a_n\mathds{1}_{[a+n,a+n+1)}(x)
\text{ \ with }a\in[0,1),\ a_n\ge 0,\ \sum_{n\in Z}a_n=1\end{equation}
\begin{equation}
\text{the symmetric triangular distribution on }[c,d]\text{ wth }d-c\in\N
\end{equation}\begin{equation}
\text{the density }f(x)=\begin{cases}x^2&\text{if } 0\le x\le 1\\
1-(x-1)^2&\text{if } 1\le x\le 2\\
0&\text{otherwise}.\end{cases}\end{equation}
A generic description of random variables $Z$ with $Z\bmod{t}\sim U[0,t)$
is given by James V.~Whittaker in \cite[Th.~~2, Th.~3]{scaleinvariant},
where the condition that $Z\bmod{t}\sim U[0,t)$ is referred to as
$Z$ being scale-invariant with $t$ in its spectrum.

For a given random variable $X$, the set
\begin{equation}\label{spectrum}
\{b\in(1,\infty)\colon X\text{ is }b\text{-Benford}\}
\end{equation}
is known to be bounded from above,
e.g. \cite{schatte1981}, 
\cite[p.~392]{baseb}, \cite[Lemma 1]{scaleinvariant},
\cite{lolbert},
or our Theorem \ref{boundedspectrum}.
This puts another restraint on the notion of base-invariance:
no set of numbers can be close to satisfying Benford's Law
for arbitrarily large bases.

The most basic property of the set (\ref{spectrum}),
cf.~\cite[Lemma 1]{scaleinvariant},
is that
\begin{equation}
X\text{ is }b\text{-Benford }\implies X\text{ is }\sqrt[n]{b}
\text{-Benford for all }n\in\N.\end{equation}
Indeed, using (\ref{bZ}), if $X=b^Z$ with $Z\bmod{1}\sim U[0,1)$
then \begin{equation}\label{aswddd4ed}
X=(\sqrt[n]{b})^{nZ}\text{ with }
nZ\bmod{1}\sim U[0,1)\end{equation}
and thus $X$ is $\sqrt[n]{b}$-Benford
for each $n\in\N$.

Let us take a closer look at this fact
to see what we can conclude about the distribution of
$aX\bmod{1}$ for $a\in(0,\infty)$
assuming that $X\bmod{1}\sim U[0,1)$.

\begin{proposition}\label{lebinv}
If $X\bmod{1}\sim U[0,1)$ then
$nX\bmod{1}\sim U[0,1)$ for all $n\in\N$.
\end{proposition}
\begin{proof}
For each $a\in(0,\infty)$,
let $T_a\colon[0,1)\to[0,1)$ be given by $T_a(x)=ax\bmod{1}$.
Notice that for each $x\in\R$ we have the equivalence
$a\in\N\iff T_a(x\bmod{1})=ax\bmod{1}$.
Let $m=\lfloor a\rfloor$, so that $m\in\N$ and $m\le a<m+1$.
Then \begin{equation}
T_a(x)=\sum_{k=0}^{m-1}\mathds{1}_{\big[\frac{k}{a},\frac{k+1}{a}\big)}(x)\big(ax-k\big)
+\mathds{1}_{\big[\frac{m}{a},1\big)}(x)\big(ax-m\big).
\end{equation}
Thus for any $t\in[0,1]$
\begin{equation}\label{rozpiska}
T_a^{-1}([0,t])=\bigcup_{k=0}^{m-1}\Big[\frac{k}{a},\frac{k+t}{a}\Big]
\cup\begin{cases}
\Big[\frac{m}{a},\frac{m+t}{a}\Big]&\colon t\le a-m\\
\Big[\frac{m}{a},1\Big)&\colon a-m\le t
\end{cases}\end{equation}
and consequently
\begin{equation}\label{husfcsrw}
\lambda\big(T_a^{-1}([0,t])\big)=
\begin{cases}
\Big(\frac{m+1}{a}\Big)t&\colon t\le a-m\\
\Big(\frac{m}{a}\Big)t+\Big(1-\frac{m}{a}\Big)&\colon a-m\le t,
\end{cases}\end{equation}
where $\lambda$ denotes the Lebesgue measure.

A look at (\ref{husfcsrw}) reveals that
$$a\in\N\ \ \iff\ \ \lambda\big(T_a^{-1}([0,t])\big)=t
\text{ \ for all }t\in[0,1],$$
but we will not be able to exploit the $\impliedby$ implication.

It is crucial to keep in mind the following elementary fact
of modulo arithmetic\begin{equation}\label{crucialhere}
a\in\N\ \ \iff\ \  T_a(x\bmod{1})=ax\bmod{1}
\text{ \ for all }x\in\R.\end{equation}
In the next line we write the same number in three different ways
using the notation introduced but we also
exploit the assumption $X\bmod{1}\sim U[0,1)$:
\begin{equation}\label{sgsfsfsdereds}
P\big(T_a(X\bmod{1})\le t\big)=
P\big(X\bmod{1}\in T_a^{-1}([0,t])\big)
=\lambda\big(T_a^{-1}([0,t])\big).\end{equation}
Now, if $a=n\in\N$ then
$$P(nX\bmod{1}\le t)=P(T_n(X\bmod{1})\le t)=
\lambda\big(T_n^{-1}([0,t])\big)=t$$
for all $t\in[0,1]$ and so $nX\bmod{1}\sim U[0,1)$ for any $n\in\N$.

However, if $a\not\in\N$ then
--- because of (\ref{crucialhere}) ---
we cannot use (\ref{sgsfsfsdereds})
to establish a relation between $P(aX\bmod{1}\le t)$
and (\ref{husfcsrw}),
and thus nothing can be concluded
about the distribution of $aX\bmod{1}$.
\end{proof}

\begin{theorem}\label{hut6tfdfecbh}
Let $X\colon\Omega\to(0,\infty)$ and $b>1$.
If $X$ is $b$-Benford then
\begin{equation}\label{tfdf5fd3fd}X\textnormal{ is }
\sqrt[n]{b}\textnormal{-Benford for all }n\in\N\end{equation}
and \begin{equation}\label{fd5ddh8d8}X^n\textnormal{ is }b
\textnormal{-Benford for all }n\in\N.\end{equation}
\end{theorem}
\begin{proof}By (\ref{aswddd4ed}),
(\ref{tfdf5fd3fd}) is a consequence of Proposition \ref{lebinv}.
The identity $$S_{\sqrt[n]{b}}(x)=S_b(x^n)$$
shows that (\ref{tfdf5fd3fd}) and (\ref{fd5ddh8d8}) are equivalent.
\end{proof}

There are two ways of looking at this result.
First, the single random variable $X$
satisfies Benford's law simultaneously
for each base from the geometric sequence $\{\sqrt[n]{b}\}_{n=1}^\infty$.
Alternatively, in the sequence of random variables $X^n$
each of them satisfies Benford's law for the same base $b$.

Let us consider a converse to Theorem \ref{hut6tfdfecbh}.
The property of a random variable $X$ that
\begin{equation}\label{hillsdef}S_b(X)\sim S_b(X^n)
\text{ for all }n\in\N\end{equation}
or equivalently
\begin{equation}S_b(X)\sim S_{\sqrt[n]{b}}(X)
\text{ for all }n\in\N\end{equation}
is used by Theodore T.~Hill
to characterize Benford's law for a given base $b$,
Theorem \ref{hills}.
Compare (\ref{rozpiska}) with (\ref{fdfedfet})
\begin{equation}\label{fdfedfet}
S_b(x^n)\le b^t\ \ \iff\ \ 
b^{\frac{k}{n}}\le S_b(x)\le b^{\frac{k+t}{n}}
\text{ for some }k\in\{0,1,\ldots,n-1\}
\end{equation}
to see that his definition
of base-invariance \cite[Def.~3.1]{baseinvariance}
is equivalent to (\ref{hillsdef}).
It should be kept in mind that the condition expressed in this definition
is a statement not only about $X$ but also about $b$.
There is no way a base-invariance of this kind for some $X$ and the base 10
can be inferred about the same $X$ and base 8,
as can be seen from the trivial example $X=10^{U[0,1)}$.

\begin{theorem}[Theodore T.~Hill]\label{hills}
Let the random variable $X\colon\Omega\to(0,\infty)$ satisfy 
$$S_b(X)\sim S_b(X^n)
\textnormal{ \ \ for all }n\in\N$$
or equivalently
$$S_b(X)\sim S_{\sqrt[n]{b}}(X)
\textnormal{ \ \ for all }n\in\N$$
for a given $b\in(1,\infty)$.
Then $$P(S_b(X)\le b^t)=p+(1-p)t\textnormal{ \ \ for all }t\in[0,1],$$
where $p=P(S_b(X)=1)$.
\end{theorem}
\begin{proof}
Putting $Y=\log_bX$, the assumption translates into
$$Y\bmod{1}\sim nY\bmod{1}\text{ \ \ for all }n\in\N.$$
See the original \cite{baseinvariance}
or a later version \cite[Section 4.3]{basictheory}.
\end{proof}

In this context, Theodore T.~~Hill discusses
Furstenberg's Conjecture \cite[p.~58]{basictheory},
but he does not remark that if $X$ is assumed to have a density
then it is enough that the condition $S_b(X)\sim S_b(X^n)$
holds for just one $n>1$:

\begin{theorem}\label{furstenberg}
Let $X$ be a real-valued continuous random variable with an arbitrary density.
If $X\bmod{1}\sim nX\bmod{1}$ for some $n\in\N$, $n>1$,
then $X\bmod{1}\sim U[0,1)$.
\end{theorem}
\begin{proof}
By \cite[Theorem 4.17]{basictheory}, for every $s\in[0,1)$,
$$\lim_{k\to\infty}P\big(n^kX\bmod{1}\le s\big)=s.$$
On the other hand, for any $k\in\N$,
$$P\big(n^kX\bmod{1}\le s\big)=P\big(X\bmod{1}\le s\big).$$
\end{proof}

The assumption that $X$ has a density cannot be altogether
dropped to achieve full generality but it can be substantially relaxed
--- see the remark on page 40 in \cite{basictheory}.
It would be nice to generalize Theorem \ref{furstenberg}
so that $n\in\N$ could be replaced with any $a>1$
to achieve
\begin{equation}\label{aXU}
X\bmod{1}\sim aX\bmod{1}\implies
X\bmod{1}\sim U[0,1).
\end{equation}
Unfortunately,
we cannot hope to generalize the method used in the proof of
\cite[Theorem 4.17]{basictheory} because in the very first line
it requires the fact that $$ax\bmod{1}=a(x\bmod{1})\bmod{1},$$
which is valid only for $a\in\N$.
Fortunately, we have an alternative proof
which also relaxes the assumption
that $X$ has a density.

\begin{theorem}\label{indistribution}
Let $X$ be a real-valued continuous random variable with an arbitrary density.
Then $aX\bmod{1}\to U[0,1)$ in distribution as $a\to\infty$
through all real values.
\end{theorem}
\begin{proof}
Let $f$ be the density of $X$.
Then the characteristic function of $X$ is given by
$\varphi(t)=\int_{-\infty}^\infty e^{itx}f(x)dx$.
By the Riemann–Lebesgue lemma,
\cite[Lemma 3 on p.~513]{feller}, \begin{equation}\label{rajchman}
\lim_{|t|\to\infty}\varphi(t)=0.\end{equation}

Fix $k\in\Z\setminus\{0\}$. By (\ref{modulosmarts}),
$$\varphi_{aX\bmod{1}}(2\pi k)=\varphi_{aX}(2\pi k)=\varphi(2\pi ka),$$
which converges to $0$ as $a\to\infty$.

To finish the proof apply the criterion (cf.~\cite[p.~35]{basictheory},
\cite[Lemma 1]{ross})
that for any random variables $Z_n,Z\colon\Omega\to[0,1)$,
$Z_n\to Z\text{ in distribution as }n\to\infty$
if and only if
$$\lim_{n\to\infty}\varphi_{Z_n}(2\pi k)=\varphi_Z(2\pi k)
\text{ \ \ for all }k\in\Z.$$
See also an analogous proof in \cite{lolbert}.
\end{proof}

Compare Theorem \ref{indistribution} with \cite[Th.~~3]{appliedfourier},
where $a$ is restricted to integers and the density of $X\bmod{1}$
belongs to $L^2[0,1]$.

Also notice that the assumption that $X$ has a density in Theorem
\ref{indistribution} can be weakened to just the condition (\ref{rajchman}),
which is related to the notion of {\em Rajchman probability},
cf.~~\cite[p.~40]{basictheory}.
The following Corollary \ref{powers} and
Theorems \ref{wniosekmocny} and \ref{baseinvchar}
can also be restated to require only this weaker condition.

\begin{corollary}[cf.~\cite{lolbert}]\label{powers}
Let $X\colon\Omega\to(0,\infty)$ be a random variable with a density.
Then for any base $b>1$,
$$S_b(X^a)\to b^{U[0,1)}\textnormal{ in distribution}$$
as $a\to\infty$ through all real values.\end{corollary}
\begin{proof}
It follows from Theorem \ref{indistribution},
by considering $Y=\log_bX$.
\end{proof}

\begin{theorem}\label{wniosekmocny}
Let $X$ be a real-valued continuous random variable with a density.
If $aX\bmod{1}\sim X\bmod{1}$ for some $a\in(1,\infty)$
then $X\bmod{1}\sim U[0,1)$.
\end{theorem}
\begin{proof}
Since
$a^n\bmod{1}\sim X\bmod{1}\text{ for all }n\in\N,$
it is enough apply Theorem \ref{indistribution}
to show that $a^nX\bmod{1}$ converges to $U[0,1)$
in distribution as $n\to\infty$.
\end{proof}

The case $a\in(0,1)$ can be included in Theorem \ref{wniosekmocny}
by considering $1/a>1$.
\\\\
We are ready for a kind of base-invariance characterization
of Benford's law:
if a positive continuous random variable has the same significand
distribution for two distinct bases
then it must be Benford for both bases.

\begin{theorem}\label{baseinvchar}
Let $1<b<\beta$.
Suppose that the random variable $X\colon\Omega\to(0,\infty)$
has a density and satisfies $$P(S_b(X)\le b^t)=
P(S_\beta(X)\le\beta^t)\textnormal{ \ \ for all }t\in[0,1].$$
Then $X$ is Benford for both bases $b$ and $\beta$, that is
$$P(S_b(X)\le b^t)=P(S_\beta(X)\le\beta^t)=
t\textnormal{ \ \ for all }t\in[0,1].$$
\end{theorem}
\begin{proof}
The assumption translates into
$$\log_bX\bmod{1}\sim\log_\beta X\bmod{1}.$$
Let $Y=\log_\beta X$ and $a=\log_b\beta>1$.
Then $\log_b X=aY$ and thus
$$aY\bmod{1}\sim Y\bmod{1}.$$
By Theorem \ref{wniosekmocny},
$$aY\bmod{1}\sim Y\bmod{1}\sim U[0,1).$$
Therefore, using (\ref{bZ}),
$$X=\beta^Y\text{ is }\beta\text{-Benford}$$
and $$X=b^{aY}\text{ is }b\text{-Benford}.$$
\end{proof}

Theorem \ref{baseinvchar} is very nice but it is essential
to answer the question whether there exists a random variable
satisfying its assumption apart from the case when the greater base
is an integral multiple of the smaller base, cf.~Theorem \ref{hut6tfdfecbh}.

In fact, using the same ideas both Peter Schatte \cite{schatte1981} and
James V.~Whittaker \cite[p.~267]{scaleinvariant}
constructed a random variable
that satisfies Benford's law simultaneously for all bases
in the interval $(1,b]$ for any $b>1$.

We are going to review Whittaker's approach here because
it can be seriously improved to study in full generality
the set of all bases for which a given random variable is Benford.

The following preparations are needed.
For $z>0$ and $x\in\R$ let us define
$$x\bmod{z}=x-\max\{nz\colon n\in\Z\wedge nz\le x\}.$$
Then $$x\bmod{z}=x-z\Big\lfloor\frac{x}{z}\Big\rfloor\in[0,z)$$
and
$$X/z\bmod{1}\sim U[0,1)\iff X\bmod{z}\sim U[0,z).$$

\begin{theorem}[cf.~Th.~4 in \cite{scaleinvariant}]
\label{modchar}
Let $X\colon\Omega\to\R$ be a random variable. Let $z>0$.
Then \begin{equation}\label{modchar1}
X\bmod{z}\sim U[0,z)\end{equation}
if and only if
\begin{equation}\label{modchar2}
\varphi_X\Big(\frac{2\pi n}{z}\Big)=0
\textnormal{ \ \ for all }n\in\N.\end{equation}
\end{theorem}
\begin{proof}
Note that if $U\sim U[0,1)$
then $\varphi_U(2\pi n)=0$ for all $n\in\N$
and recall that, by Theorem \ref{coefficients},
the distribution of any random variable $Y\colon\Omega\to[0,1)$
is uniquely determined by its Fourier coefficients $\varphi_Y(2\pi n)$,
$n\in\N$.

Notice that $\varphi_{X/z}(t)=\varphi_X(t/z)$ for all $t\in\R$.
Thus the equivalnce (\ref{modchar1})$\iff$(\ref{modchar2})
is restated as
$$X/z\bmod{1}\sim U[0,1)\ \ \iff\ \ \varphi_{X/z}(2\pi n)=0
\text{ for all }n\in\N,$$
which, by (\ref{modulosmarts}), turns into
$$X/z\bmod{1}\sim U[0,1)\ \ \iff\ \ \varphi_{X/z\bmod{1}}(2\pi n)=0
\text{ for all }n\in\N,$$
which is the just-mentioned characterization of the uniform
distribution $U[0,1)$ in terms of its Fourier coefficients.
\end{proof}

A weaker version of Theorem \ref{modchar}
is given by James V.~Whittaker \cite[Th.~4]{scaleinvariant},
where (\ref{modchar1})$\implies$(\ref{modchar2})
is proved in a different way and the converse
(\ref{modchar2})$\implies$(\ref{modchar1})
is also proved in a different way with the additional assumption
that $X$ has a density with bounded variation.
In fact, all his calculations are made redundant
by appealing to the fact that
$\varphi_X(2\pi n)=\varphi_{X\bmod{1}}(2\pi n)$,
see (\ref{modulosmarts}).

As a side remark notice that Theorem \ref{modchar}
has an interesting corollary that enriches
\cite[Lemma 4(b) on p.~501]{feller}
by adding the information that $\varphi(t)\not=0$ for all $t\in\R$,
because otherwise --- $|\varphi|$ being periodic --- the
discrete variable under consideration would be uniformly distributed
modulo $z$ for some $z>0$.

Let us turn Theorem \ref{modchar} into a tool
with which to characterize the set of all bases
for which a given random variable is Benford.
This criterion was given by Whittaker in 1983
and then by Lolbert in 2008, \cite{lolbert}.
In our version $X$ is not assumed to have a density,
so it is fully general.

\begin{theorem}[Whittaker's Criterion]\label{charbases}
Let $X\colon\Omega\to(0,\infty)$ be a random variable.
Let $b>1$. Then
\begin{equation}\label{charbases1}
X\textnormal{ is }b\textnormal{-Benford}\end{equation}
if and only if\begin{equation}\label{charbases2}
\varphi_{\ln X}\Big(\frac{2\pi n}{\ln b}\Big)=0
\textnormal{ \ \ for all }n\in\N.\end{equation}
\end{theorem}
\begin{proof}
Since $\log_bX=\ln X/\ln b$, $$(\ref{charbases1})
\iff\log_bX\bmod{1}\sim U[0,1)\iff
\ln X\bmod{\ln b}\sim U[0,\ln b)\iff(\ref{charbases2}),$$
the last equivalance by Theorem \ref{modchar}.
\end{proof}

As a side remark, let us record the following nice corollary to
Theorem \ref{charbases},
which is an alternative proof of (\ref{spectrum}).

\begin{theorem}\label{boundedspectrum}
Let $X\colon\Omega\to(0,\infty)$ be a random variable.
Then $$\{b\in(1,\infty)\colon X\textnormal{ is }b
\textnormal{-Benford}\}$$
is bounded from above.
\end{theorem}
\begin{proof}
Let $\varphi$ be the characteristic function of $Z=\ln X$.
Since $\varphi(0)=1$ and $\varphi$ is continuous,
there is an $r>0$ such that $\varphi(t)\not=0$ for all $t\in(0,r)$.
If $X$ is $b$-Benford for some $b>1$,
then by Theorem \ref{charbases}, $\varphi(2\pi/\ln b)=0$.
Hence $2\pi/\ln b>r$ and finally $b<\exp(2\pi/r).$
\end{proof}

The following theorem was proven by Peter Schatte in 1981
\cite{schatte1981} and then in 1983 by James V.~Whittaker
\cite{scaleinvariant}. Although the core idea is the same,
they each have a different approach, which follows from
their earlier results, so it is likely that they have arrived
at this solution independently.
Whittaker's way involves the intermediate step of
Theorem \ref{charbases}, which is important in its own right.
Schatte refers directly to Theorem \ref{minusx},
which is exactly the same method but without explicitly
stating what I call Whittaker's Criterion, Theorem \ref{charbases}.
Moreover, Schatte gives more details how to construct a characteristic function
with compact support while Whittaker merely cites Feller's textbook.

\begin{theorem}[Peter Schatte, James V.~Whittaker]\label{intervalspectrum}
For any $\beta>1$ there is a random variable that is Benford
simultaneously for all bases in $(1,\beta]$.
\end{theorem}\begin{proof}
By \cite[p.~503]{feller},
for each $a>0$, the function $f\colon\R\to[0,\infty)$ given by
$$f(x)=\frac{1-\cos ax}{\pi ax^2}$$
is the probability density function of a random variable
$Z\colon\Omega\to\R$ whose characteristic function is given by
$$\varphi(t)=\begin{cases}1-\frac{|t|}{a}&\colon|t|\le a\\
0&\colon|t|\ge a.\end{cases}$$

Let $a=2\pi/\ln b$. Then
$\varphi_Z(t)=0$ for all $t\ge 2\pi/\ln\beta$.
Thus $$\varphi_Z(2\pi n/\ln b)=0
\text{ \ \ for all }b\in(1,\beta],n\in\N.$$
By Theorem \ref{charbases}, the random variable $X=e^Z$
is Benford for all $b\in(1,\beta]$.
\end{proof}

This super-Benford example is a mathematician's answer
to a theoretical problem. But the canonical examples of Benford random variables, which might be expected to approximate real-world phenomena,
are of the form $$X=b^{U[c,d)}$$ with densities
$$f(x)=\frac{1}{(d-c)\ln b}\mathds{1}_{[b^c,b^d)}(x)
\frac{1}{x}.$$
Each of them has the largest base $\beta$ for which it is Benford
and all its other Benford bases are of the form $\sqrt[n]{\beta}$.
However, it is possible that choosing an incompatible base
for the significand analysis of a random variable that is Benford 
for its native base may produce a very close approximation to being Benford
for the incompatible base.
These ideas are summarized in Theorem \ref{bhcbvtgtgve}.

\begin{lemma}\label{g5r5fef25ft}
Let $X\sim U[c,d)$ with $c<d$. Then
$X\bmod{1}\sim U[0,1)\iff d-c\in\N$ and
$$\sup_{A\in\borel([0,1))}\big|P(X\bmod{1}\in A)
-\lambda(A)\big|\le\frac{1}{d-c},$$
where $\lambda$ is the Lebesgue measure
on the Borel subsets of $[0,1)$, $\borel([0,1))$.
\end{lemma}\begin{proof}Let $n=\lfloor d-c\rfloor$.
Then $$\frac{n\lambda(A)}{d-c}\le P(X\bmod{1}\in A)\le\frac{(n+1)\lambda(A)}{d-c}$$
because the set $A\subset[0,1)$ fits exactly $n$ times modulo 1 inside the interval $[c,d)$.
\end{proof}

\begin{theorem}\label{bhcbvtgtgve}
Let $X=\beta^{U[c,d)}$ with $1<\beta$ and $c<d$.
Then
$$\{b\in(1,\infty)\colon X\textnormal{ is }b
\textnormal{-Benford}\}=\{\beta^{\frac{d-c}{n}}\colon n\in\N\}$$
and for all $b>1$,
$$\sup_{A\in\borel([0,1))}
\big|P(\log_bX\bmod{1}\in A)-\lambda(A)\big|\le
\frac{1}{(d-c)\log_b\beta}.$$
\end{theorem}
\begin{proof}
Apply Lemma \ref{g5r5fef25ft} to
$\log_bX=U[c,d)\log_b\beta=U[c\log_b\beta,d\log_b\beta)$.
\end{proof}

\begin{definition}[cf.~\cite{scaleinvariant}]\label{bspectrum}
Let $X\colon\Omega\to(0,\infty)$ be a random variable.
Its {\em Benford spectrum}, $B_X$, is defined as the set of those bases
$b>1$ for which $X$ is $b$-Benford.\end{definition}

\begin{problem}Characterize the sets $B_X$.\end{problem}

We already know that $B_X$ can be any interval $(1,b]$,
Theorem \ref{intervalspectrum},
and any discrete sequence $\{\sqrt[n]{b}\colon n\in\N\}$,
$b>1$, Theorem \ref{bhcbvtgtgve}.
The following theorem can be used to construct other examples.
Note that we have gotten rid of Whittaker's constraint
that $\ln XY$ has a density with bounded variation.

\begin{theorem}[cf.~last page in \cite{scaleinvariant}]
\label{gsgsgsgsgs}
Let $X,Y\colon\Omega\to(0,\infty)$ be independent random variables.
Then $B_X\cup B_Y\subset B_{XY}$.
\end{theorem}
\begin{proof}
This is a direct consequence of Whittaker's Criterion,
Theorem \ref{charbases}, because
since $\ln X,\ln Y$ are independent, the characteristic function of
$\ln X+\ln Y=\ln XY$
is the product of their respective characteristic functions, (\ref{phixy}).
\end{proof}\begin{proof}
Alternatively, this is a direct consequence of Theorem \ref{filozof},
because if $b\in B_X$ then $X=b^Z$ with $Z\bmod{1}\sim U[0,1)$
and $Y=b^W$ with $Z$ and $W$ being independent,
so $XY=b^{Z+W}$ with $(Z+W)\bmod{1}\sim U[0,1)$
and thus $b\in B_{XY}$.
\end{proof}

\section{Concluding remarks}

We should clearly distinguish between Benford's Law
as a kind of universal observation on the one hand and the particular
mathematical probability distribution possessed by
certain random variables on the other,
which is variously called Benford's law, Benford's property,
or logarithmic distribution.

Benford's law is about the logarithms of numbers to a chosen base
$b\in(1,\infty)$ being uniformly distributed modulo 1.
This translates into the so-called logarithmic distribution
of the significand on the interval $[1,b)$.
A set of numbers or a random variable may have
varying degrees of conformance to this logarithmic distribution
of its significand depending on the choice of base.

There is no mathematical reason to consider only integer bases.
The focus on first digits is also arbitrary
and should be considered only as a convenient choice of a histogram.

Once a base is chosen, scale-invariance seems to be the fundamental
property of the logarithmic distribution of the significand.
It is satisfied in the most general case when the random variable
under consideration is multiplied by any independent variable.
Moreover, the logarithmic distribution of the significand can be
characterized as the only one that is invariant under
multiplication by just one arbitrary
continuous independent random variable.

In contrast to scale-invariance,
it is a mistaken notion that a given random variable
must be Benford for all bases once it is Benford for a single base.
Many restraints on this notion have been discussed.
The perceived base-invariance of Benford's Law
as a stitistical observation must be due to its approximating
character because strictly speaking we should expect
conformance to Benford's Law to be dependent on the choice
of base. In practice, in certain natural conditions,
this dependence may be negligible within certain bounds.

\end{document}